\documentclass[11pt,twoside]{amsart}

\usepackage{amsmath}
\usepackage{amsthm}
\usepackage{amsfonts, amssymb}
\usepackage{mathrsfs}
\usepackage[all]{xy}
\usepackage{url}
\usepackage{graphics}
\usepackage{epstopdf}

\usepackage{latexsym}
\usepackage[dvips]{graphicx}

\theoremstyle{plain}
\newtheorem{lema}{Lemma}[section]
\newtheorem{prop}[lema]{Proposition}
\newtheorem{teo}[lema]{Theorem}

\newtheorem*{intro1}{Theorem \ref{main1}}

\newtheorem{coro}[lema]{Corollary}
\theoremstyle{remark}

\newtheorem{obs}[lema]{Remark}

\theoremstyle{definition}
\newtheorem{defi}[lema]{Definition}
\newtheorem{ej}[lema]{Example}

\newcommand{\kker}{\textrm{ker}}
\newcommand{\Deck}{\textrm{Deck}}
\newcommand{\Fix}{\textrm{Fix}}

\newcommand{\ee}{\textbf{E}}
\newcommand{\hh}{\mathcal{H}}
\newcommand{\kp}{\mathcal{K}}
\newcommand{\x}{\mathcal{X}}

\renewcommand{\langle}{[}
\renewcommand{\rangle}{]}

\def\Z{\mathbb{Z}}

\begin{document}

\title[$\boldsymbol{G}$-colorings of posets]{$\boldsymbol{G}$-colorings of posets, coverings and presentations of the fundamental group}

\author[J.A. Barmak]{Jonathan Ariel Barmak}
\author[E.G. Minian]{Elias Gabriel Minian}

\address{Departamento  de Matem\'atica--IMAS\\
 FCEyN, Universidad de Buenos Aires\\ Buenos
Aires, Argentina}

\email{jbarmak@dm.uba.ar}
\email{gminian@dm.uba.ar}

\begin{abstract}
We introduce the notion of a coloring of a poset, which
consists of a labeling of the edges in its Hasse
diagram by elements in a given group $G$. We use $G$-colorings
to describe the covering maps of posets and present a new method based on colorings to obtain concrete and simple presentations of the fundamental group of polyhedra.
\end{abstract}



\maketitle

\section{Introduction} \label{sectionintro}

 The classical edge-path group $\mathcal{E}(K,v_0)$ of a simplicial complex $K$
describes combinatorially the fundamental group of $K$
in terms of paths in its $1$-skeleton. This concept can be translated
into the context of posets resulting in a description of $\pi _1
(X,x_0)$ by edge-paths in the Hasse diagram of the poset $X$ \cite{BM1}. In this article we use such a description just as a
starting point. We introduce the notion of a coloring of a poset, which
is a labeling $\ee (X) \to G$ of the edges in the Hasse
diagram of $X$ by elements in a given group $G$, and use $G$-colorings
to study covering maps. As a consequence we obtain a concrete presentation of the fundamental group of
$X$, more suitable than the description given by
edge-paths. Our classification of coverings in terms of $G$-colorings 
provides a new insight into the theory of coverings of polyhedra and it is the key point in this theory. 

 There is a well-known and close relationship between the homotopy
theory of polyhedra and partially ordered sets. To
each simplicial complex $K$ one can associate the face poset $\x (K)$
and for each poset $X$ one can construct the order complex $\kp (X)$.
The combinatorics of posets can be used to study topological properties of complexes by means of these two functors. Examples of this
interaction are Quillen's work on the poset of $p$-subgroups of a
finite group \cite{Qui} and Chari's approach to Forman's discrete
Morse theory \cite{Cha,For} (see also \cite{Koz, Min}).  In the same direction,
the interplay between the combinatorics of posets and the topology of
polyhedra has been used in \cite{BM} to investigate simple homotopy
types of complexes and in \cite{Bar3} to give an alternative proof and
applications of Quillen's Theorem A for posets. Any poset can be seen
as a topological space, more precisely as an Alexandroff space (or
$A$-space for short), without necessity of using the functors $\x$ and
$\kp$: the open sets of $X$ are its order ideals. McCord proved that
the topology of such spaces is closely related to the topology of their associated complexes $\kp (X)$. 
Concretely, there is weak homotopy equivalence $\kp(X)\to
X$ and in particular these two spaces have the same homology and homotopy
groups \cite{Mcc}.  A poset can also be regarded as a category with at
most one morphism between any two objects (the order complex $\kp(X)$
is just the classifying space of the category $X$). This provides an
alternative way to understand the connection between topological and combinatorial
properties.

In section \ref{colorings} we characterize regular coverings of posets in terms of colorings.
The class of \textit{admissible} and \textit{connected}
$G$-colorings plays an important role in
this theory. These colorings classify the
normal subgroups of the fundamental group of the poset whose quotients  are
isomorphic to $G$. Concretely, we prove the following result.

\begin{intro1}
Let $X$ be a connected locally finite poset, $x_0\in X$ and $G$ a group. There
exists a correspondence between the set of equivalence classes of
admissible connected $G$-colorings of $X$ and the set of normal
subgroups $N\triangleleft \pi_1 (X,x_0)$ such that $\pi_1(X,x_0)/ N$
is isomorphic to $G$.
\end{intro1}

In particular there is a direct connection between $G$-colorings of
$X$ and equivalence classes of regular coverings of $X$ with deck
transformation group isomorphic to $G$.  In Theorem \ref{main2} we
give  the explicit construction of the corresponding covering.

In Section \ref{presentacion} we use colorings to find alternative presentations of the fundamental group. We exhibit various examples and applications of this new characterization. For instance, we deduce a generalization of van Kampen's theorem (Theorem \ref{nuestrovk}). We also characterize, in terms of colorings, the posets with abelian fundamental group.

In Section \ref{maps} we use colorings to study maps between the fundamental groups and in the last section of the paper we consider a
combinatorial problem related with boards on surfaces.

\section{Preliminaries}

In this section we recall the basic notions on $A$-spaces, their relationship with posets and simplicial complexes, and the description of their fundamental group in terms of edge-paths. For more details we refer the reader to \cite{Bar2,BM1,Mcc,Sto}.

A preorder is a set with a reflexive and transitive relation. Such a set is a poset if the relation is also antisymmetric. An $A$-space is a topological space in which arbitrary intersections of open sets are open. Finite topological spaces and, more generally, locally finite spaces, are examples of $A$-spaces. A locally finite space is a topological space in which every point has a finite neighborhood. There is a natural correspondence between $A$-spaces and preorders. Given an $A$-space $X$, for each point $x$ in $X$ let $U_x$ be the intersection of all the open sets containing $x$. This is the smallest open set which contains $x$. The preorder associated to the $A$-space $X$ has the same underlying set and the relation is given by $x\le y$ if $x\in U_y$. Conversely, given a preorder $\le$ on a set $X$, the topology corresponding to this relation is the one generated by the subsets $U_x=\{y\in X \ | \ y\le x\}$, for every $x\in X$. A function between $A$-spaces is continuous if and only if it is order-
preserving. Note that if $X$ is an $A$-space, $\{U_x\}_{x\in X}$ is a basis for the topology. Any $A$-space is locally contractible since the sets $U_x$ are contractible. It is easy to see that there is a homotopy which is the identity for $t<1$ and it is the constant $x$ for $t=1$. In particular, any $A$-space has a universal cover. Given an $A$-space $X$, the closed sets of $X$ form another topology on the underlying set of $X$, called the opposite topology. The preorder associated to this topology is the opposite order of $X$. This space is denoted by $X^{op}$. Note that a map $f:X\to Y$ between $A$-spaces is continuous if and only if the induced map $f^{op}:X^{op}\to Y^{op}$, which coincides with $f$ in the underlying sets, is continuous. If $X$ is an $A$-space, the closure of a point $x$ in $X$ is denoted by $F_x$. Note that $F_x^{X}=\{y\in X,\ x\leq y \}=(U_x^{X^{op}})^{op}$. The notations $F_x^{X}$ and $U_x^{X}$ will be used when we need to emphasize the space $X$ where these subsets are 
considered. The \textit{star} of a point $x$ in an $A$-space $X$ is $C_x=U_x\cup F_x$. We denote respectively $\hat{U}_x$, $\hat{F}_x$ and $\hat{C}_x$ the reduced sets $U_x\smallsetminus \{x\}$, $F_x\smallsetminus \{x\}$ and $C_x\smallsetminus \{x\}$.

Recall that a topological space $X$ is said to be $T_0$ if for any two points $x,y\in X$ there is an open set which contains one and only one of them. This is the unique separation axiom that we will work with. Note that if an $A$-space is $T_1$ (i.e. if each point is closed), then it is discrete. It is not hard to prove that an $A$-space is $T_0$ if and only if the corresponding preorder is a poset. 

A finite $T_0$-space is a finite poset. A locally finite $T_0$-space is a \textit{locally finite poset}, i.e. a poset such that for every element $x$ there are only finitely many elements smaller than $x$. The Hasse diagram of a locally finite $T_0$-space $X$ is the digraph whose vertices are the points of $X$ and whose edges are the pairs $(x,y)$ such that $x\prec y$. Here $x\prec y$ means that $x$ is covered by $y$, i.e. $x<y$ and there is no $z\in X$ such that $x<z<y$. In the graphical representation of the Hasse diagram, instead of drawing the edge $(x,y)$ with an arrow, we simply put $y$ over $x$ (see for example Figure \ref{proy}). Note that a map $f:X\to Y$ between locally finite $T_0$-spaces is continuous if and only if $x\prec x'$ implies $f(x)\le f(x')$. 

In contrast to the case of finite simplicial complexes, it is easy to decide whether two finite spaces are homotopy equivalent or not. The combinatorial description of the homotopy types of finite spaces is due to Stong \cite{Sto}. Given a finite $T_0$-space $X$, a point $x\in X$ is called a \textit{beat point} if it covers a unique element or if it is covered by a unique element. It follows immediately from Stong's ideas that a finite $T_0$-space is contractible if and only if it is possible to remove beat points one by one from $X$ to obtain the space of one point $*$. Moreover, removing a beat point $x$ from a finite $T_0$-space $X$ produces a subspace $X\smallsetminus \{x\}$ homotopy equivalent to $X$. Concretely, one has the following

\begin{prop} {\rm (Stong \cite[Theorem 2]{Sto})} \label{stong}
If $x$ is a beat point of a finite $T_0$-space $X$, $X\smallsetminus \{x\}$ is a strong deformation retract of $X$.
\end{prop}

In particular, if $X$ is contractible and we remove beat points $x_1, x_2, \ldots, x_n$, one by one, the subspace $Y$ obtained in this way is also contractible, so we can continue removing beat points to obtain the singleton. Contractible finite $T_0$-spaces correspond to \textit{dismantlable} posets.

The order complex $\kp (X)$ of a $T_0$-$A$-space $X$ is the simplicial complex whose simplices are the non-empty finite chains of $X$. The polyhedron $\kp (X)$ and the $A$-space $X$ do not have in general the same homotopy type, however they do have isomorphic homotopy and homology groups. Moreover, McCord proved \cite{Mcc} that there exists a weak homotopy equivalence $\mu _X: \kp (X) \to X$ (i.e. a continuous map which induces isomorphisms in all the homotopy groups). A continuous map $f:X\to Y$ between $T_0$-$A$-spaces has an associated simplicial map $\kp (f):\kp (X)\to \kp (Y)$ such that $\kp (f)\mu _X=\mu _Y f$. In the other direction, if $K$ is a simplicial complex, or more generally a regular CW-complex,  the face poset $\x (K)$ is the $T_0$-$A$-space which corresponds to the poset of cells of $K$ ordered by inclusion. In this case there exists a weak homotopy equivalence $K\to \x (K)$.
  
It is well-known that the fundamental group of a simplicial complex can be described by means of the edge-path group (see \cite[Section 3.6]{Spa} for more details). The fundamental group of a locally finite $T_0$-space can be described in a similar way. This was developed in \cite{BM1} for finite $T_0$-spaces, but it extends straightforward to locally finite $T_0$-spaces. Let $X$ be a locally finite $T_0$-space. The set of edges of the Hasse diagram of $X$ will be denoted by $\ee (X)$, an \textit{edge-path} from $x$ to $y$ in $X$ is a sequence $(x_0, x_1)(x_1, x_2) \ldots  (x_{n-1}, x_n)$ of ordered pairs such that $(x_i, x_{i+1})\in \ee (X)$ or $(x_{i+1}, x_i)\in \ee (X)$ for every $0\le i<n$ and such that $x_0=x$, $x_n=y$. Note that since $X$ is locally finite, the following statements are equivalent: (1) $X$ is a connected topological space, (2) $X$ is path-connected and (3) for any two points $x,y \in X$ there exists an edge-path from $x$ to $y$. Of course, an edge-path $\xi$ from $x$ to $y$ and an edge-path $\xi '$ from $y$ to $z$ can be concatenated to form an edge-path $\xi \xi'$ from $x$ to $z$. The inverse of an edge-path $\xi=(x_0, x_1) (x_1, x_2) \ldots  (x_{n-1}, x_n)$ is defined as $\xi ^{-1}=(x_{n}, x_{n-1}) (x_{n-1},x_{n-2}) \ldots  (x_{1}, x_0)$. An edge-path $(x_0, x_1) (x_1, x_2) \ldots  (x_{n-1}, x_n)$ is \textit{monotonic} if $(x_i, x_{i+1})\in \ee (X)$ for all $i$ or if $(x_{i+1},x_i)\in \ee (X)$ for all $i$. When the concatenations $\xi=\xi_1 \xi_2 \xi_3 \xi_4$ and $\xi '=\xi_1 \xi_4$ are well-defined, $\xi$ and $\xi '$ are said to be \textit{elementary equivalent} if $\xi _2$ and $\xi_3$ are monotonic. This relation generates an equivalence relation of edge-paths from $x$ to $y$. The class of an edge-path $\xi$ from $x$ to $y$ is denoted by $\langle \xi \rangle$. Given $x_0\in X$ we denote by $\hh (X,x_0)$ the group whose elements are the classes of closed edge-paths at $x_0$, i.e. the edge-paths from $x_0$ to $x_0$, and the product is defined by $\langle \xi \rangle \langle \xi ' \rangle 
= \langle \xi \xi ' \rangle$. Note that this is well defined and the identity $\langle \rangle$ is the class of the empty edge-path. The inverse $\langle \xi \rangle ^{-1}$ of $\langle \xi \rangle$ is $\langle \xi ^{-1} \rangle$.

Note that if $\xi$ and $\xi '$ are two monotonic edge-paths from $x$ to $y$, $\xi _1$ is an edge-path from $x_0$ to $x$ and $\xi_2$ is an edge-path from $y$ to $x_0$ then $\langle \xi _1 \xi \xi_2 \rangle =\langle \xi_1 \xi ' \xi_2 \rangle$.

The group $\hh (X,x_0)$ and the edge-path group $\mathcal{E}(\kp (X), x_0)$ of the simplicial complex $\kp (X)$ are isomorphic. The isomorphism $\phi _X : \hh (X,x_0) \to \mathcal{E}(\kp (X), x_0)$ is defined in \cite{BM1} (see also \cite[pp.24]{Bar2}). An explicit isomorphism $\epsilon _X :\mathcal{E}(\kp (X), x_0)\to \pi_1 (\kp (X), x_0)$ is described in \cite[pp.136]{Spa}. In particular $\hh (X,x_0)$ is isomorphic to $\pi _1 (X,x_0)$ via the isomorphism $\eta _X=(\mu _X)_* \epsilon _X \phi _X : \hh (X,x_0) \to \pi _1 (X, x_0)$. Concretely $\pi _1(X,x_0)$ is isomorphic to the set of closed edge-paths at $x_0$ where two closed edge-paths are equivalent if we can obtain one from the other by replacing a monotonic sub-edge-path by another monotonic edge-path with the same origin and end and where the inverse of an edge-path is given by the edge-path in the opposite direction.

The application $\hh$ is functorial. If $f: X\to Y$ is a continuous map between locally finite $T_0$-spaces and $\xi=(x_0, x_1) (x_1, x_2) \ldots  (x_{n-1}, x_n)$ is a closed edge-path at $x_0$ in $X$, there is a closed edge-path $\xi '$ at $f(x_0)$ in $Y$ which is obtained by concatenation of monotonic edge-paths from $f(x_i)$ to $f(x_{i+1})$ for every $i$. We define $\hh (f)([\xi])=f_*([\xi])=[\xi ']$. It is easy to check that $f_*=\hh (f):\hh (X,x_0) \to \hh (Y, f(x_0))$ is a well defined homomorphism. Moreover, the application $\phi$ above is a natural isomorphism between $\hh$ and $\mathcal{E} \kp$. In particular we have the following

\begin{obs} \label{naturalidad}
Let $f:X\to Y$ be a continuous map between locally finite $T_0$-spaces. Then there is a commutative diagram where the horizontal arrows are isomorphisms
\begin{displaymath}
\xymatrix@C=30pt{ \hh (X,x_0) \ar@{->}^{\phi _X}[r] \ar@{->}^{f_*}[d] & \mathcal{E}(\kp (X),x_0) \ar@{->}^{\epsilon _X}[r] \ar@{->}^{\kp (f)_*}[d] & \pi_1 (\kp (X),x_0) \ar@{->}^{ (\mu _X)_*}[r] \ar@{->}^{\kp (f)_*}[d] & \pi _1 (X, x_0) \ar@{->}^{f_*}[d] \\
								\hh (Y,f(x_0)) \ar@{->}^{\phi _Y}[r] & \mathcal{E}(\kp (Y),f(x_0)) \ar@{->}^{\epsilon _Y}[r] & \pi_1 (\kp (Y),f(x_0)) \ar@{->}^{ (\mu _Y)_*}[r]  & \pi _1 (Y, f(x_0)). \\ }
\end{displaymath}
\end{obs}

If $B$ is locally finite and $T_0$ and $p:E\to B$ is a covering, then $E$ is also locally finite and $T_0$. In particular if $b_0\in B$ and $e_0\in p^{-1}(b_0)$, the fundamental groups of $E$ and $B$ can be described with the groups $\hh (E,e_0)$ and $\hh (B,b_0)$. In this case the map $p_*: \hh (E, e_0) \to \hh (B, b_0)$ is easy to describe. If $\xi=(e_1,e_2)(e_2,e_3) \ldots (e_{r-1},e_r)$ is an edge-path in $E$, then $p_*(\xi)= (p(e_1),p(e_2))(p(e_2),p(e_3)) \ldots (p(e_{r-1}),p(e_r))$ is also an edge-path in $Y$. The covering $p$ maps edges to edges since $p|_{U_e}:U_e\to U_{p(e)}$ (also $p|_{F_e}:F_e\to F_{p(e)}$) is a homeomorphism for every $e\in E$.
The homomorphism $p_*: \hh (E,e_0)\to \hh (B,b_0)$ is given by $p_*(\langle \xi \rangle)=\langle p_* (\xi) \rangle$. Given an edge-path $\xi$ in $B$ starting in $b_0$, there exists a unique edge-path $\widetilde{\xi}$ in $E$ starting in $e_0$ such that $p_*(\widetilde{\xi})=\xi$. If $\xi$ and $\xi '$ are two equivalent edge-paths 
from $b_0$ to a point $b_1$, then it is clear that both lifts $\widetilde{\xi}$ and $\widetilde{\xi '}$ end in the same point.

The group $\hh Fix(e_0)=p_*(\hh (E,e_0))\leqslant \hh (B,b_0)$ consists of the classes of closed edge-paths at $b_0$ which lift to closed edge-paths at $e_0$.

\begin{prop} \label{epfix}
Let $B$ be a locally finite $T_0$-space, $b_0\in B$, $p:E\to B$ a covering and $e_0\in p^{-1}(b_0)$. The isomorphism $\eta _B=(\mu _B)_* \epsilon _B \phi _B: \hh (B,b_0)\to \pi _1 (B,b_0)$ restricts to an isomorphism $\hh Fix(e_0)\to \Fix(e_0)=p_*(\pi_1(E,e_0))$. 
\end{prop}
\begin{proof}
It follows immediately from the commutativity of the diagram in Remark \ref{naturalidad}.
\end{proof} 


\section{$G$-colorings and regular coverings}\label{colorings}


In this section we introduce the notion of a coloring of a locally finite poset $X$, which is used to classify the normal subgroups of $\pi_1 (X)$. In particular, the colorings of $X$ describe all its regular coverings.

\begin{defi}
Let $X$ be a connected locally finite $T_0$-space and let $G$ be a group. A \textit{$G$-coloring of $X$} is a map $c: \ee (X) \to G$. If $c$ is a $G$-coloring of $X$ and $(x,y)\in \ee (X)$, we define $c(y,x)=c(x,y)^{-1}$. Given a $G$-coloring $c$, there is an induced \textit{weight} map $w$ (also denoted by $w_c$), which associates an element of $G$ to every edge-path of $X$. This map is defined by $$w((x_0, x_1)(x_1, x_2) \ldots (x_{n-1}, x_n))=\prod\limits_{i=0}^{n-1} c(x_i, x_{i+1})=c(x_0,x_1)c(x_1,x_2)\ldots c(x_{n-1},x_n). $$ The weight of the empty edge-path is defined as $1$, the identity of $G$.
  
A $G$-coloring of $X$ is \textit{admissible} if for any $x\le y$ in $X$ and any two monotonic edge-paths $\xi$, $\xi '$ from $x$ to $y$, the weights $w(\xi)$ and $w (\xi ')$ are equal. Let $x_0\in X$. An admissible $G$-coloring $c$ of $X$ induces a group homomorphism $W=W_c:\hh (X, x_0)\to G$ defined by $W(\langle \xi \rangle)=w(\xi)$.  

A $G$-coloring is said to be \textit{connected} if for every $g\in G$ there exists a closed edge-path at $x_0$ whose weight is $g$. When the $G$-coloring is admissible, this is equivalent to saying that $W: \hh (X,x_0)\to G$ is an epimorphism. Note that this definition is independent of the choice of the base point $x_0$. The motivation of the term ``connected" for such a coloring is the space $E(c)$ which appears in Theorem \ref{main2}.
\end{defi}

\begin{defi} \label{defequi}
Two $G$-colorings $c,c'$ of $X$ are said to be \textit{equivalent} if there exists  an automorphism $\varphi : G\to G$ and an element $g_x\in G$ for each $x$, such that $$c'(x,y)=\varphi (g_x c(x,y) g_y^{-1})$$ for every $(x,y)\in \ee (X)$. In this case we write $c\sim c'$.  
\end{defi}

It is easy to see that this is an equivalence relation in the set of $G$-colorings of $X$. If $c\sim c'$ and $\xi$ is an edge path from $x$ to $y$ in $X$ then, with the notation of the last definition,  $w_{c'}(\xi)=\varphi (g_x w_c(\xi) g_y^{-1})$. Therefore, if $c\sim c'$ and $c$ is admissible, then so is $c'$. Also, if $c$ is connected, so is $c'$.

\begin{teo} \label{main1}
Let $X$ be a connected locally finite poset, $x_0\in X$ and $G$ a group. There exists a correspondence between the set of equivalence classes of admissible connected $G$-colorings of $X$ and the set of normal subgroups $N\triangleleft \pi_1 (X,x_0)$ such that $\pi_1(X,x_0)/ N$ is isomorphic to $G$.  
\end{teo}
\begin{proof}

Since the groups $\hh (X,x_0)$ and $\pi _1(X, x_0)$ are naturally isomorphic, it suffices to prove the result for $\hh(X,x_0)$.
An admissible connected $G$-coloring $c$ of $X$ induces an epimorphism $W: \hh (X,x_0) \to G$. Then $G$ is isomorphic to $\hh (X,x_0)/N$ where $N=\kker(W)$. Equivalent colorings induce the same subgroup $N$ since the associated weight maps differ in an automorphism of $G$.

Conversely, if $N \triangleleft \hh (X,x_0)$ is such that $\hh (X,x_0)/N \simeq G$, then choose an isomorphism $\psi : \hh (X,x_0)/N \to G$ and define $\rho=\psi p : \hh (X,x_0)\to G$ where $p: \hh (X,x_0) \to \hh (X,x_0)/N$ is the canonical projection. Since $X$ is connected and locally finite, for each $x\in X$ there exists an edge-path $\gamma _x$ from $x_0$ to $x$. Given $(x,y)\in \ee (X)$, define $c(x,y)=\rho (\langle \gamma _x (x,y) \gamma _y^{-1} \rangle)$. If $\xi$ and $\xi '$ are two monotonic edge-paths in $X$ from a point $x$ to a point $y\ge x$, then $w_c(\xi)=\rho (\langle \gamma _x \xi \gamma _y ^{-1} \rangle)=\rho (\langle \gamma _x \xi ' \gamma _y ^{-1} \rangle)=w_c (\xi ')$, since $\langle \gamma _x \xi \gamma _y ^{-1} \rangle = \langle \gamma _x \xi ' \gamma _y ^{-1} \rangle$. Therefore the $G$-coloring $c$ is admissible. The induced morphism $W_c: \hh (X,x_0) \to G$ is the composition of $\rho$ with the conjugation by $\langle \gamma _{x_0} \rangle$. It follows that $W_c$ is an epimorphism 
and therefore $c$ is connected. Note that different choices of the isomorphism $\rho$ and the edge-paths $\gamma _x$ induce equivalent colorings.  

It remains to show that these constructions are reciprocal. Let $c$ be an admissible connected $G$-coloring of $X$ and let $N=\kker(W)$ be the induced normal subgroup of $\hh (X,x_0)$. We can choose the isomorphism $\psi :\hh (X,x_0)/N \to G$ to be the morphism induced by $W$ in the quotient. In this way, the map $\rho : \hh (X,x_0) \to G$ coincides with $W$. Therefore the new color $c' (x,y)$ of an edge $(x,y)\in \ee (X)$ is $W(\langle \gamma _x (x,y) \gamma _y^{-1} \rangle)=w_c(\gamma _x)c(x,y)w_c(\gamma _y)^{-1}$. Thus, $c '\sim c$.

Finally, if $N\triangleleft \hh (X,x_0)$ induces a coloring $c$, then the kernel of $W_c$ is $\kker(\rho)=N$.
\end{proof}

\begin{coro} \label{colorepi}

Let $X$ be a connected locally finite poset, $x_0\in X$ and $G$ a group. Then, there exists an admissible connected $G$-coloring of $X$ if and only if there exists an epimorphism $\pi _1(X,x_0) \to G$.
\end{coro}

Let $B$ be a path-connected, locally path-connected, semilocally simply-connected space, $p:E\to B$ a covering, $b_0\in B$ and $e_0\in p^{-1}(b_0)$. We denote by $\Deck(p)$ the group of deck transformations (=covering transformations) of $p$, and let $\Fix (e_0)=p_*(\pi _1 (E, e_0)) \leqslant \pi _1 (B, b_0)$. Recall that two coverings $p:E\to B$, $p':E'\to B$ are said to be equivalent if there exists a homeomorphism $h:E\to E'$ such that $p'h=p$. The correspondence between conjugacy classes of subgroups of $\pi _1 (B,b_0)$ and equivalence classes of coverings of $B$ maps a normal subgroup $N\triangleleft \pi_1 (B,b_0)$ to a regular covering $p:E\to B$ with $\Fix (e_0)=N$ (see \cite[Section 1.3]{Hat}). In this case, $\Deck(p)$ is isomorphic to $\pi _1 (B,b_0)/N$. Therefore we deduce the following     

\begin{coro} \label{corresprev}
Let $B$ be a connected locally finite poset and let $G$ be a group. There exists a correspondence between the set of equivalence classes of regular coverings $p:E\to B$ of $B$ with $\Deck(p)$ isomorphic to $G$ and the set of equivalence classes of admissible connected $G$-colorings of $B$. 
\end{coro}

We state a more precise version of Corollary \ref{corresprev} making an explicit construction of the covering associated to a given $G$-coloring. Given an admissible connected $G$-coloring $c$ of $B$, we define the poset $E=E(c)= \{(x,g) \ | \ x\in B, \ g\in G \}$ with the relations $(x,g)\prec (y, gc(x,y))$ whenever $x\prec y$ in $B$.

\begin{teo} \label{main2}
Let $B$ be a connected locally finite poset and let $G$ be a group. If $c$ is an admissible connected $G$-coloring of $B$, then the projection $p(c): E(c)\to B$ onto the first coordinate is a regular covering of $B$ with $\Deck(p)$ isomorphic to $G$. Moreover, if $c$ and $c'$ are equivalent admissible connected $G$-colorings of $B$, then $p(c)$ and $p(c')$ are equivalent coverings of $B$. This application describes a correspondence between the set of equivalence classes of admissible connected $G$-colorings of $B$ and equivalence classes of regular coverings of $B$ with deck transformation group isomorphic to $G$.    
\end{teo}
\begin{proof}
The map $p=p(c):E=E(c)\to B$ is clearly continuous. We claim that if $b\in B$, then $$p^{-1}(U_b)=\coprod\limits_{g\in G} U_{(b,g)},$$ and that the restrictions to each $U_{(b,g)}$ are homeomorphisms. The inclusion $U_{(b,g)}\subseteq p^{-1}(U_b)$ follows from the continuity of $p$. Now suppose $(b',g)\in p^{-1}(U_b)$. Then $b'\le b$ and there exists a chain $b'=b_1\prec b_2 \prec \ldots \prec b_r=b$. Since $(b_i,h)\prec (b_{i+1}, hc(b_i, b_{i+1}))$, we have that $$(b',g)\le (b,gc(b_1, b_2)c(b_2, b_3) \ldots c(b_{r-1},b_r)).$$ Thus $(b',g)\in U_{(b,h)}$ for $h=gc(b_1, b_2)c(b_2, b_3) \ldots c(b_{r-1},b_r)$.

Suppose $(b',h)\in U_{(b,g_1)}\cap U_{(b,g_2)}$. Then there exists a chain $b'=b_1\prec b_2 \prec \ldots \prec b_r=b$ such that $g_1=hc(b_1, b_2)c(b_2, b_3) \ldots c(b_{r-1},b_r)$ and there is a chain $b'=b_1'\prec b_2' \prec \ldots \prec b_s'=b$ such that $g_2=hc(b_1', b_2')c(b_2', b_3') \ldots c(b_{r-1}',b_s')$. By the admissibility of $c$, $g_1=g_2$. This proves that the union is disjoint.

The map $U_b\to U_{(b,g)}$ which maps $b'$ into $(b',gc(b_r,b_{r-1})^{-1}\ldots c(b_2,b_1)^{-1})$, where $b'=b_1\prec b_2 \prec \ldots \prec b_r=b$ is any chain between $b'$ and $b$, is a continuous inverse of $p|_{U_{(b,g)}}$. Therefore $p$ is a covering.

Note that $E$ is a connected space since the coloring $c$ is connected. If $\xi$ is a closed edge-path at $b_0$ with weight $w_c(\xi)=g$, then the lift of $\xi$ from $(b_0,1)$ ends in $(b_0,g)$. Therefore the connectedness of $c$ implies that any two points in the fiber of $b_0$ lie in the same component of $E$.

Suppose $c'\sim c$, that is, there exists $\varphi \in \textrm{Aut} (G)$ and $g_b\in G$ for every $b\in B$ such that $c'(b_1,b_2)=\varphi (g_{b_1} c(b_1,b_2) g_{b_2}^{-1})$ for each $(b_1,b_2)\in \ee (B)$. Consider the map $h:E(c)\to E(c')$ which maps $(b,g)$ to $(b,\varphi(gg_b^{-1}))$. If $(b_1,g) \prec (b_2,gc(b_1,b_2))$ then $$h(b_1,g)=(b_1, \varphi (gg_{b_1}^{-1}))\prec (b_2, \varphi (gg_{b_1}^{-1})c'(b_1,b_2))=$$ $$=(b_2, \varphi (gc(b_1,b_2)g_{b_2}^{-1}))=h(b_2, gc(b_1,b_2)).$$

Hence, $h$ is continuous and $p(c')h=p(c)$. Moreover $h':E(c')\to E(c)$ given by $h'(b,g)=(b,\varphi ^{-1}(g)g_b)$ is the inverse of $h$. Therefore $p(c)$ and $p(c')$ are equivalent coverings.

Let $b_0\in B$. Note that a closed edge-path $\xi$ at $b_0$ lifts to a closed edge-path at $(b_0, 1)\in E(c)$ if and only if $w_c(\xi)=1\in G$. Therefore, $\hh Fix(b_0,1)=\kker (W_c)$. On the other hand, the application which associates a normal subgroup of $\pi _1(B,b_0)$ to an admissible connected $G$-coloring of $B$, maps $c$ into $\eta _B (\kker(W_c))\triangleleft \pi _1 (B, b_0)$. This subgroup corresponds to a regular covering of $B$ whose fix subgroup is equal to $\eta _B(\kker(W_c))$, or equivalently by Proposition \ref{epfix}, to a covering with $\hh Fix$ equal to $\kker (W_c)$. Therefore, the composition of the correspondence of Theorem \ref{main1} with the correspondence between normal subgroups of $\pi _1(B,b_0)$ and regular coverings of $B$, is the application described above. In particular, $p(c):E(c)\to B$ is a regular covering with $\Deck(p(c))$ isomorphic to $G$ and this assignation is a one-to-one correspondence.

\end{proof}

One can prove that the functor $\x$ induces a correspondence between the equivalence classes of coverings of a simplicial complex $K$ and the equivalence classes of coverings of $\x (K)$. A similar result holds for the functor $\kp$ (see \cite{BM3}). Therefore, colorings can be used to describe all the regular coverings of a given polyhedron.

\begin{ej} \label{planoproy}
The poset $X$ of Figure \ref{proy} is the face poset of a regular CW-complex homeomorphic to the real projective plane $\mathbb{R}P^2$. Therefore, its fundamental group is the group $\mathbb{Z}_2$ of order two. We will show in Section \ref{presentacion} an alternative way to compute $\pi _1 (X,x_0)$ (ignoring the fact that this poset is related to the projective plane).

\begin{figure}[h] 
\begin{center}
\includegraphics[scale=0.6]{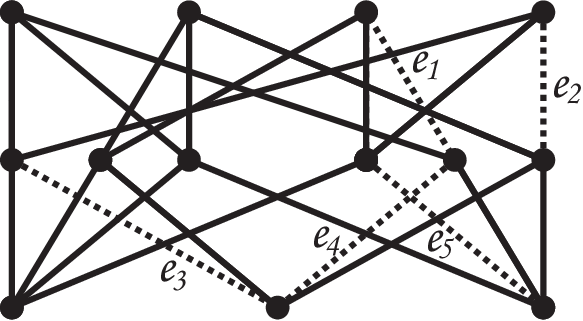}
\caption{A finite model of the projective plane.}\label{proy}
\end{center}
\end{figure}

Consider the $\mathbb{Z}_2$-coloring of $X$ in which every solid edge of Figure \ref{proy} is colored with the identity $0$ of $\mathbb{Z}_2$ and where the four dotted edges are colored with the non-trivial element of $\mathbb{Z}_2$. It is easy to check that this coloring is admissible and connected and corresponds, by Theorem \ref{main1}, to a subgroup $N\triangleleft \pi _1(X,x_0)$ such that $\pi _1 (X,x_0) /N$ is isomorphic to $\mathbb{Z}_2$. Therefore $N$ is the trivial group and the corresponding covering is the universal cover. Now, it is easy to distinguish the closed edge-paths which are trivial in $\hh (X,x_0)$ once we have the coloring corresponding to the universal cover. A closed edge-path $\xi$ at $x_0$ is trivial if and only if it lifts to a loop in the universal cover. This happens if and only if its weight $w_c(\xi)$ is trivial. Therefore, in this example a closed edge-path represents the identity of $\hh (X,x_0)$ if and only if it passes through a dotted edge an even number of times.
\end{ej}

\begin{ej}[Detecting $K(G,1)$'s]
A topological space having a universal cover is an Eilenberg-MacLane space $K(G,1)$ if and only if its universal cover is homotopically trivial, i.e. weak homotopy equivalent to the singleton. We use Theorem \ref{main2} to construct a covering from a given coloring, and the fact that, in the context of posets, sometimes it is easy to recognize homotopically trivial spaces via beat points (see Proposition \ref{stong}).

Consider the space $X$ of Figure \ref{moeb1} with the following $\mathbb{Z}$-coloring $c$. The solid edges are colored with the trivial element $0\in \mathbb{Z}$ and the dotted edges are colored with the generator $1\in \mathbb{Z}$. This is an admissible and connected $\mathbb{Z}$-coloring of $X$. 

\begin{figure}[h] 
\begin{center}
\includegraphics[scale=0.4]{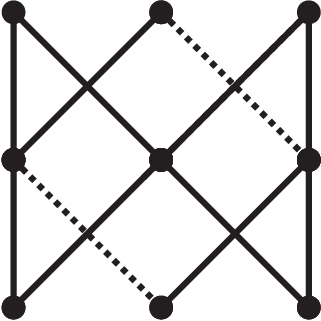}
\caption{The $\mathbb{Z}$-coloring $c$ of the poset $X$.}\label{moeb1}
\end{center}
\end{figure}

The covering $E$ associated to this coloring according to Theorem \ref{main2} is sketched in Figure \ref{moeb2}. 

\begin{figure}[h] 
\begin{center}
\includegraphics[scale=0.4]{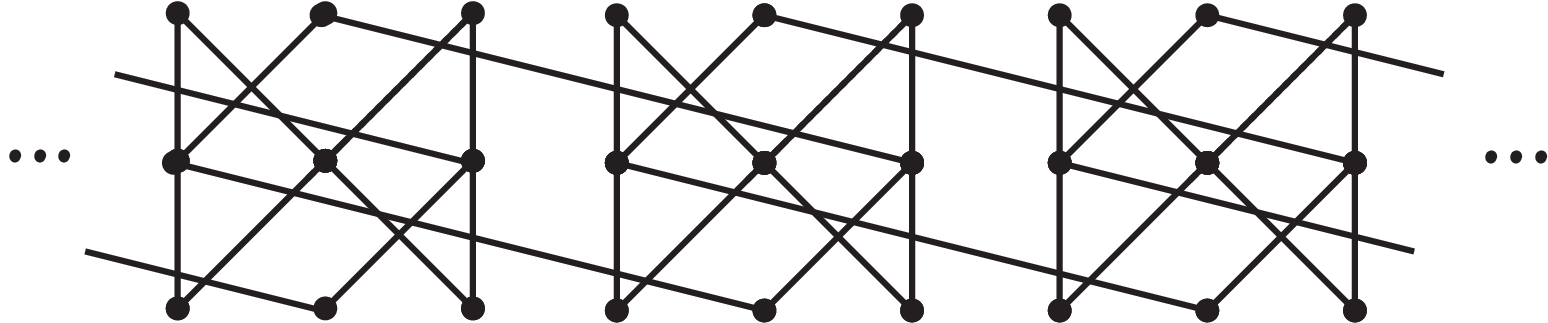}
\caption{The covering $E$ associated to $c$.}\label{moeb2}
\end{center}
\end{figure}

We claim that $E$ is a homotopically trivial space. Indeed, the Hasse diagram of $E$ is a countable union of copies $X_n$, $n\in \mathbb{Z}$, of the diagram $X_0$ in Figure \ref{moeb3}. The intersection of $X_n$ and $X_m$ has two points if $|n-m|=1$ and is empty otherwise. The space $X_0$ is contractible. Moreover, the subspace of two points, $x$ and $y$, is a deformation retract of $X_0$. This is really easy to check, removing beat points one by one.

\begin{figure}[h] 
\begin{center}
\includegraphics[scale=0.4]{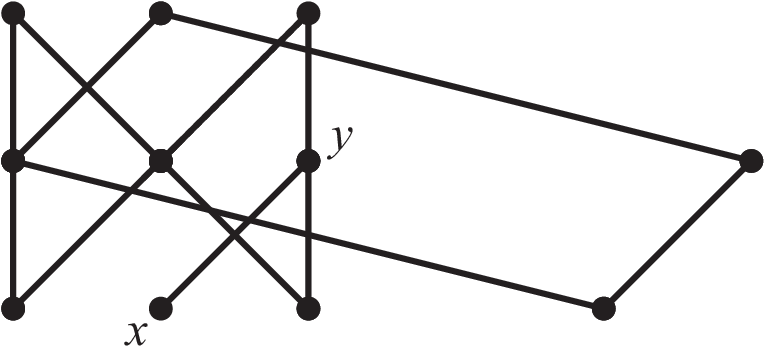}
\caption{The space $X_0$ which deformation retracts to the edge $(x,y)$.}\label{moeb3}
\end{center}
\end{figure}   

This shows in fact that $Z_{n,k+1}=X_n\cup X_{n-1}\cup \ldots \cup X_{n-k-1}$ deformation retracts to $Z_{n,k}=X_{n}\cup X_{n-1}\cup \ldots \cup X_{n-k}$, and then $Z_{n,k}$ is contractible for any $n\in \mathbb{Z}$, and $k\ge 0$. Now, any compact subspace of $E$ is contained in a subspace $Z_{n,k}$ since any minimal open set $U_z$ of $E$ intersects finitely many copies of $X_0$ (one or two). Then the image of any map from a sphere to $E$ is contained in a contractible subspace, which proves that $E$ is homotopically trivial. In particular $\pi _r(X)=0$ for every $r\ge 2$. This proves that $X$ is a $K(G,1)$ for some $G$. In fact, $X$ is a $K(\mathbb{Z},1)$. One can easily verify that $\pi_1(X)=\mathbb{Z}$ using for example Theorem \ref{main3} below.

\end{ej}

Note that an admissible $G$-coloring $c:\ee (X)\to G$ is equivalent to a functor $X\to G$ where $G$ is viewed as a category with a unique object and one arrow for each element of the group. Since every morphism in the category $G$ is an isomorphism, a functor $X\to G$ is equivalent to a functor $\Sigma^{-1} X \to G$ from the category of fractions of $X$, which is obtained from $X$ by formally inverting all the arrows (see \cite{GZ}). This is equivalent to a group homomorphism $\pi _1 (X,x_0)\to G$ (cf. \cite[pp.89-90]{Qu}). 

\bigskip

To finish this section we exhibit a method for constructing a poset with fundamental group isomorphic to any given group. This idea copies, in some sense, Milnor's classical construction of universal bundles and classifying spaces of groups \cite{Mil}.

Let $G$ be a group. Let $X$ be the following poset of height $2$. The set of minimal elements is $G \times \Z _3$. The set of points of height $1$ is $G\times G \times \Z _3$ and the set of maximal points is $G\times G \times G$. The order is given as follows $(g,h,i+1)$ covers $(g,i)$ and $(h,i+2)$, and $(g,h,k)$ covers $(g,h,1)$, $(k,g,2)$ and $(h,k,0)$ for each $g,h,k\in G$ and $i\in \Z _3$. The group $G$ acts on $X$ by left multiplication in each coordinate belonging to $G$. This action is properly discontinuous and therefore the projection $p:X\to X/G$ is a covering with deck transformation group isomorphic to $G$. The space $X$ is simply-connected. This can be proved for instance by induction in the order of $G$, using Theorem \ref{main3} of the next section. Therefore $X/G$ is a poset with fundamental group isomorphic to $G$.

For $G=\Z _2$ this construction gives a space $X/G$ of $13$ points, isomorphic to the model of the projective plane of Example \ref{planoproy}.  

\section{Presentations of the fundamental group} \label{presentacion}

Let $X$ be the poset of Figure \ref{ocho}. Let $G=\mathbb{Z}*\mathbb{Z}$ be the free group on two generators $g, h$. There is an admissible connected $G$-coloring which is trivial in the solid edges and such that the two dotted edges are colored one with $g$ and the other with $h$. By Corollary \ref{colorepi}, there exists an epimorphism $\pi _1(X,x_0)\to G$. Again by Corollary \ref{colorepi}, there exists an admissible connected $\pi _1(X,x_0)$-coloring $c$ of $X$. Since the undirected subgraph given by the solid edges is a tree, it is possible to show that there is a coloring $c'$ of $X$ which is equivalent to $c$ and which is trivial in the solid edges. Hence, $\pi _1 (X,x_0)$ is generated by two elements, the $c'$-colors of the dotted edges. This says that there exists an epimorphism $G\to \pi _1 (X, x_0)$. One can deduce then that $\pi _1 (X,x_0)$ is isomorphic to $\mathbb{Z}*\mathbb{Z}$.

\begin{figure}[h] 
\begin{center}
\includegraphics[scale=0.45]{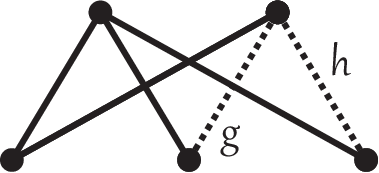}
\caption{A poset whose fundamental group is the free group on two generators.}\label{ocho}
\end{center}
\end{figure}

In general it is not true that the existence of epimorphisms $G\to H$ and $H\to G$ implies that $G$ and $H$ are isomorphic (see for instance \cite{BS}). Two posets admitting the same set of groups $G$ for which there is an admissible connected $G$-coloring, need not have isomorphic fundamental groups. Nevertheless we will see that we can use colorings to compute the fundamental groups of posets (and therefore of simplicial complexes).

By a subdiagram of a Hasse diagram $D$ we mean a subgraph of $D$. If $X$ is a locally finite poset, any subdiagram of the Hasse diagram of $X$ is the Hasse diagram of a locally finite space $A$. This space need not be a subspace of $X$. However, the inclusion $A\hookrightarrow X$ is continuous.

\begin{lema} \label{arbol}
Let $G$ be a group with identity $1$. Let $X$ be a connected locally finite $T_0$-space and let $D$ be a subdiagram of the Hasse diagram of $X$ which corresponds to a connected space $A$. If the map $i_*:\hh (A, x_0)\to \hh (X,x_0)$ induced by the inclusion is trivial for some $x_0\in A$, then for each admissible $G$-coloring $c$ of $X$ there exists a $G$-coloring $c'$ equivalent to $c$ such that $c'(x,y)=1$ for every $(x,y)\in \ee (A)$. In particular, this holds when $A$ is simply-connected.
\end{lema}

\begin{proof}
Choose an edge-path $\gamma _a$ in $A$ from $x_0$ to $a$ for each $a\in A$. Define the $G$-coloring $c'$ of $X$  by $c'(x_1,x_2)=g_{x_1}c(x_1,x_2)g_{x_2}^{-1}$ where $g_{x}=w_{c}(\gamma _x)$ if $x\in A$ and $g_x=1$ if $x\notin A$. Then $c'$ and $c$ are equivalent $G$-colorings of $X$ and $c'$ restricted to $A$ is trivial. Given $(a_1,a_2)\in \ee (A)$, one has $c'(a_1,a_2)=w_{c}(\gamma _{a_1})c(a_1,a_2)w_{c}(\gamma _{a_2})^{-1}=w_c(\gamma _{a_1}(a_1,a_2)\gamma _{a_2}^{-1})$ which is $1$ since the closed edge-path $\gamma _{a_1}(a_1,a_2)\gamma _{a_2}^{-1}$ is equivalent to the empty path at $x_0$ by the hypothesis on $i_*:\hh (A, x_0)\to \hh (X,x_0)$. 
\end{proof}
\begin{obs} \label{obspi2}
Lemma \ref{arbol} can be generalized as follows. If $\{D_j\}_{j\in J}$ is a collection of pairwise disjoint connected subdiagrams of $X$ and the inclusions $D_j\hookrightarrow X$ induce the trivial homomorphism on fundamental groups, then for each admissible $G$-coloring $c$ of $X$ there exists an equivalent coloring which is trivial in all the diagrams $D_j$ simultaneously. To prove this we follow the proof of Lemma \ref{arbol} above choosing the edge-paths $\gamma_a$ carefully. Let $x_0$ be any point of $X$. Choose a point $x_j$ in each diagram $D_j$ and an edge-path $\gamma _j$ in $X$ from $x_0$ to $x_j$. For each point $a\in D_j$ let $\gamma _a'$ be an edge-path in $D_j$ from $x_j$ to $a$ and let $\gamma _a=\gamma _j\gamma _a'$. Define the coloring $c'$ and the $g_x$ as before considering $A=\bigcup D_j$. If $(a_1,a_2)\in \ee (D_j)$, $\gamma _{a_1}(a_1,a_2)\gamma _{a_2}^{-1}$=$\gamma _j \gamma _{a_1}'(a_1,a_2)(\gamma _{a_2}')^{-1} \gamma _j ^{-1}$ is 
equivalent 
to the empty path at $x_0$ since, by hypothesis, $\gamma _{a_1}'(a_1,a_2)(\gamma _{a_2}')^{-1}$ is equivalent to the empty path at $x_j$. 
\end{obs}

\begin{defi}
Let $X$ be a connected locally finite $T_0$-space and let $x_0\in X$. Choose for each $x\in X$ with $x\neq x_0$ an edge-path $\gamma _x$ from $x_0$ to $x$ and take $\gamma _{x_0}$ to be the trivial edge-path. The \textit{standard coloring} of $X$ is the $\hh (X,x_0)$-coloring  given by $c(x,y)=\langle \gamma _x (x,y) \gamma _y^{-1} \rangle$. Clearly $c$ is admissible and connected since the weight of a closed edge-path $\xi$ at $x_0$ is $\langle \xi \rangle$. If we take a different $\gamma _x$ for $x\neq x_0$, we obtain an equivalent coloring. Therefore, the standard coloring of $X$ is well defined up to equivalence.   
\end{defi}

The following is the main result of this section. Although at first sight its statement may seem technical, the examples below show that it can be easily applied to compute the fundamental group of posets. 

\begin{teo} \label{main3}
Let $X$ be a connected locally finite $T_0$-space and let $x_0\in X$. Let $D$ be a subdiagram of the Hasse diagram of $X$ which corresponds to a simply-connected space $A$. Let $\{e_\alpha \}_{\alpha \in \Lambda}$ be the subset of $\ee (X)$ of edges which are not in $D$. Let $G$ be the group generated by the $e_\alpha$'s with the relations given by admissibility. Concretely, for any two chains $$x=x_1\prec x_2\prec \ldots \prec x_r=y, $$ $$x=x_1'\prec x_2'\prec \ldots \prec x_s'=y $$ from any point $x$ to any point $y$, we put a relation $$\prod\limits_{(x_i,x_{i+1})\notin D} (x_i,x_{i+1}) = \prod\limits_{(x_i',x_{i+1}')\notin D} (x_i',x_{i+1}').$$ 
Suppose there is a subset $\Gamma\subseteq \Lambda$ such that the classes $\{\overline{e}_\alpha\}_{\alpha \in \Gamma}$ generate $G$ and such that for each $\alpha \in \Gamma$ there exists a closed edge-path $\omega_\alpha$ in $x_0$ which contains $e_\alpha$ exactly once and contains no other edge $e_\beta$ for $\beta \in \Lambda$. Then $\pi _1(X,x_0) \simeq G$.  
\end{teo}
\begin{proof}
We construct first a $G$-coloring $\hat c$ of $X$. We color all the edges in $D$ with $1\in G$ and each edge $e_\alpha$ with $\overline{e}_\alpha$. This coloring is admissible by definition of $G$. Let $W_{\hat c}:\hh (X,x_0) \to G$ be the weight map induced by $\hat c$. Replacing, if necessary,  $\omega_\alpha$ by $\omega_\alpha ^{-1}$,  we have for every $\alpha \in \Gamma$, $$W_{\hat c}(\langle \omega _\alpha \rangle)=w(\omega _\alpha)=\overline{e}_\alpha.$$  

Let $c$ be the standard coloring of $X$. The induced weight $W_c:\hh (X,x_0) \to \hh(X, x_0)$ is the identity. By Lemma \ref{arbol}, there exists an $\hh (X,x_0)$-coloring $c'$ of $X$ equivalent to $c$ which is trivial in $A$. Since $c'\sim c$, the weight $W_{c'}$ induced by $c'$ is $W_c$ composed with an automorphism of $\hh (X,x_0)$. Hence $W_{c'}$ is an automorphism of $\hh (X,x_0)$.

Define $\varphi : G\to \hh (X,x_0)$ by $\varphi (\overline{e}_\alpha)=c'(e_\alpha)$ for every $\alpha \in \Lambda$. This homomorphism is well-defined since $c'$ is admissible and $\{\overline{e}_\alpha\}_{\alpha \in \Gamma}$ generates $G$. Moreover, since $c'$ is connected, $\hh (X,x_0)$ is generated by $\{c'(e_\alpha)\}_{\alpha \in \Lambda}$. Therefore, $\varphi$ is an epimorphism.

Let $\alpha \in \Gamma$. Since $\omega _\alpha$ passes through $e_\alpha$ only once and all the other edges in $\omega _\alpha$ have weight $1$ with respect to the coloring $c'$, then $W_{c'}(\langle \omega _\alpha \rangle)=c'(e_\alpha)$. Thus, $$W_{\hat c}W_{c'}^{-1}\varphi (\overline{e}_\alpha)=W_{\hat c}W_{c'}^{-1}(c'(e_\alpha))=W_{\hat c}(\langle \omega _\alpha \rangle)=\overline{e}_\alpha .$$
Then $W_{\hat c}W_{c'}^{-1}\varphi$ is the identity of $G$ and, in particular, $\varphi$ is injective. Therefore $\varphi$ is an isomorphism.       
\end{proof}

Therefore, in order to compute the fundamental group of a connected locally finite poset $X$, we choose a simply-connected subdiagram $D$ of $X$ satisfying the hypotheses of Theorem \ref{main3} (for example, a simply-connected subdiagram containing all the vertices of $X$). The generators of $\pi _1(X)$ are the edges which are not in $D$ and the relators are given by \textit{digons}. A digon in a poset $X$ is a subdiagram which is the union of two different monotonic edge-paths from a point $x$ to a point $y$. Moreover, in the presentation of the group it suffices to consider only the relations given by the \textit{simple digons}, i.e. digons in which the two chains have no vertex in common with the exception of $x$ and $y$. 

\medskip

This result can be applied to compute the fundamental group of any regular CW-complex by means of its face poset. Note that for any regular CW-complex $K$, $\x(K)$ is locally finite.

\begin{ej}
Consider the poset $X$ whose Hasse diagram is shown in Figure \ref{proy}. Its edges are the solid lines together with the dotted lines.
The subdiagram given by the solid lines corresponds to a simply-connected space $A$. It is easy to check that in fact $A$ is a contractible space (to prove this, we only have to show that it is possible to reduce the space $A$ to a point by removing beat points one by one). The group $G$ of Theorem \ref{main3} is then generated by the classes of the dotted edges $e_1, e_2, e_3, e_4, e_5$. There is a digon containing $e_2$ and $e_3$ which says that $\overline{e}_2=\overline{e}_3$ is one of the relations in the presentation of $G$. There is another digon which contains the edges $e_4$ and $e_1$ producing the relation $\overline{e}_4\overline{e}_1=1$. After checking all possible digons containing at least one dotted edge, we obtain the following admissibility relations: $\overline{e}_4\overline{e}_1=1$, $\overline{e}_2=\overline{e}_3$, $\overline{e}_2=\overline{e}_5$, $\overline{e}_1=\overline{e}_5$, $\overline{e}_3=\overline{e}_4$. Therefore, the group $G$ is isomorphic to the group $\mathbb{Z}_2$, 
generated by $\overline{e}_3$. By Theorem \ref{main3}, the fundamental group of $X$ is isomorphic to $\mathbb{Z}_2$.
\end{ej}

\begin{obs}
Consider the $G$-coloring $\hat{c}$ in the proof of Theorem \ref{main3}. The color of an edge $e$ is its class $\overline{e}$ in $G$. By the proof of Theorem \ref{main3}, the weight map $W_{\hat{c}}$ associated to this coloring is an isomorphism, which implies that this coloring corresponds to the trivial subgroup of $\pi_1 (X)$. Therefore its corresponding covering $E(\hat{c})$ is the universal covering $\tilde X$ of $X$. 
\end{obs}

\begin{obs} \label{siempre}
Given a connected locally finite poset $X$, it is always possible to find a subdiagram $D$ of the Hasse diagram of $X$ in such a way that the hypotheses of Theorem \ref{main3} are fulfilled. Namely, we can take $D$ such that its underlying undirected graph is a maximal tree of the underlying undirected graph of the Hasse diagram of $X$. It is easy to see that any closed edge-path in $D$ is equivalent to the trivial edge-path. Then $A$, the locally finite space corresponding to $D$, is simply-connected. Any edge of $X$ which is not in $D$ is contained in a closed edge-path with all the other edges in $D$. Therefore, it is possible to apply the theorem using $\Gamma =\Lambda$.

In fact, the locally finite space $A$ is contractible. When $A$ is finite this is clear using Proposition \ref{stong}. When $A$ is locally finite we can use Proposition \ref{stong} together with the standard idea of the proof of \cite[Proposition 1.A.1]{Hat} and the fact that $A$ has the weak topology with respect to its edges. Moreover, any subdiagram $Y$ of $X$ is contained in a subdiagram $\widetilde{Y}$ of $X$ which contains all the points of $X$ and such that the space $Y$ is a strong deformation retract of $\widetilde{Y}$. 
\end{obs}

\begin{ej}

Consider the poset $X_1$ of Figure \ref{simpcon}. The subdiagram $D$ given by the solid edges is a tree. The remaining edges are the generators for the presentation of $\pi _1(X_1)$ in the statement of Theorem \ref{main3}. 

\begin{figure}[h] 
\begin{center}
\includegraphics[scale=0.7]{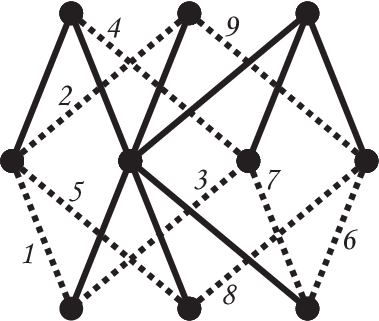}
\caption{A simply-connected space.}\label{simpcon}
\end{center}
\end{figure}

The dotted edge labeled with the number $1$ is contained in a digon with all the other edges in $D$. Therefore, the edge $1$ is the trivial element of $\pi _1(X_1)$. Edge $2$ is contained in a digon with all other edges representing the trivial element (two in $D$ and the other being edge $1$). In each step we can choose a new dotted edge contained in a digon whose edges are in $D$ or were already labeled. Therefore each dotted edge represents the trivial element and then $\pi _1(X_1)=0$.  

\bigskip

One last example $X_2$ appears in Figure \ref{piunoz}. As before, the subdiagram $D$ given by the solid edges is simply-connected and edges $1$ to $5$ represent the trivial element of $\pi _1(X_2)$.

\begin{figure}[h] 
\begin{center}
\includegraphics[scale=0.7]{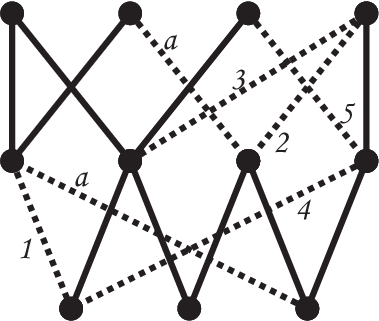}
\caption{A space with infinite cyclic fundamental group.}\label{piunoz}
\end{center}
\end{figure}

Two dotted edges remain after this process, labeled with the letter $a$. They are part of the same digon, and this relation says that they represent the same element of $\pi _1(X_2)$. None of these two edges is part of another digon, so $\pi _1(X_2)$ is the infinite cyclic group.

\end{ej}

\begin{obs}
There exists an analogue to Theorem \ref{main3} for simplicial complexes. Let $K$ be a simplicial complex. If $L$ is a simply-connected subcomplex containing all the vertices of $K$, the fundamental group of $K$ is isomorphic to the group generated by the ordered $1$-simplices of $K$ with the relations $e=1$ if $e$ is in $L$ and $e_0e_1e_2=1$ if $e_0+e_1+e_2$ is the boundary of a $2$-simplex of $K$. However by means of the poset $\x(K)$, our result allows one to manipulate the simplicial complex combinatorially (not only simplicially), resulting in better or more tractable presentations. The advantages of this discrete approach will be more clear later.
\end{obs}

\begin{coro} \label{be}
Let $X$ be a connected locally finite poset and let $B$ be a subdiagram of the Hasse diagram of $X$ which corresponds to a simply-connected space and such that any maximal chain of $X$ has all its edges in $B$ except perhaps for one. Then the fundamental group of $X$ is free.
\end{coro}
\begin{proof}
We can assume that $B$ contains all the points of $X$ by Remark \ref{siempre} and then Theorem \ref{main3} applies. Since any monotonic edge-path in $X$ has at most one edge not in $B$ then the relators of the presentation of $\pi _1(X)$ either identify two generators or identify a generator with the trivial element. Therefore, $\pi _1 (X)$ is free.
\end{proof}

\begin{teo} \label{nuestrovk}
Let $X$ be a connected locally finite $T_0$-space. Let $A$ and $B$ be two connected subdiagrams of the Hasse diagram of $X$ such that every edge of $X$ is in $A$ or $B$. Suppose that the diagram $C=A\cap B$ of common vertices and common edges is connected. Let $x_0 \in C$ and let $i: C\to A$, $j:C\to B$ be the canonical inclusions. Let $N\leq \pi _1(A,x_0)*\pi _1(B,x_0)$ be the normal subgroup generated by the words $i_*([\gamma])j_*([\gamma])^{-1}$, for every $[\gamma] \in \pi _1 (C,x_0)$. Then there exists an epimorphism $(\pi _1(A,x_0)*\pi _1(B,x_0))/N \to \pi _1(X,x_0)$. Moreover, if each simple digon of $X$ is contained in $A$ or in $B$, then $\pi _1 (X,x_0)$ is isomorphic to $(\pi _1(A,x_0)*\pi _1(B,x_0))/N$.
\end{teo}
\begin{proof}
By Remark \ref{siempre} there exists a subdiagram $D_C$ of $C$ which is simply-connected and contains all the vertices of $C$. Moreover, there exist subdiagrams $D_A$ and $D_B$ of $A$ and $B$ containing all the vertices of $A$ and of $B$ respectively, which strong deformation retract into $D_C$. Therefore $D_A\cup D_B$ is a simply-connected subdiagram of $X$ and we can apply Theorem \ref{main3} to obtain presentations of $\pi _1 (C)$, $\pi _1 (A)$, $\pi _1 (B)$ and $\pi _1 (X)$. The presentation of $\pi _1(C)$ is  $<G_C|R_C>$, where $G_C$ is the set of edges in $C$ which are not in $D_C$, and there is a relator  for each digon in $C$. The presentations of $\pi_1 (A)$, $\pi _1 (B)$ and $\pi _1 (X)$ are $<G_C\cup G_A| R_C\cup R_A>$, $<G_C\cup G_B| R_C\cup R_B>$ and $<G_C\cup G_A \cup G_B| R_C\cup R_A \cup R_B \cup R_X>$ respectively. Here $G_A$ is the set of edges of $A$ which are not in $D_A\cup C$ and the relators $R_A$ are given by digons in $A$ which are not in $C$. $G_B$ and $R_B$ are defined similarly. The relators in $R_X$ are given by digons of $X$ which are neither in $A$ nor in $B$. Note that the following diagram

\begin{displaymath}
\xymatrix@C=23pt{ <G_C|R_C> \ar@{->}^{\alpha}[r]  \ar@{->}^{\beta}[d] & <G_C\cup G_A| R_C\cup R_A> \ar@{->}[d] \\
                  <G_C\cup G_B| R_C\cup R_B> \ar@{->}[r] & <G_C\cup G_A\cup G_B| R_C\cup R_A \cup R_B> }
\end{displaymath}
in which every homomorphism maps each generator to itself, is a pushout. Moreover, $W_{\hat{c}}i_*=\alpha W_{\hat{c}}$ and $W_{\hat{c}}j_*=\beta W_{\hat{c}}$, where $W_{\hat{c}}$ denotes the three isomorphisms $\hh (C, x_0) \to <G_C|R_C>$, $\hh (A, x_0) \to <G_C\cup G_A|R_C\cup R_A>$ and $\hh (B, x_0) \to <G_C\cup G_B|R_C\cup R_B>$ constructed in the proof of Theorem \ref{main3}. By Remark \ref{naturalidad}, $<G_C\cup G_A\cup G_B| R_C\cup R_A \cup R_B>$ is isomorphic to $(\pi _1(A,x_0)*\pi _1(B,x_0))/N$.

 If each simple digon of $X$ is contained in $A$ or $B$, $(\pi _1(A,x_0)*\pi _1(B,x_0))/N = <G_C\cup G_A\cup G_B| R_C\cup R_A \cup R_B> \to <G_C\cup G_A \cup G_B| R_C\cup R_A \cup R_B \cup R_X>=\pi _1(X,x_0)$ is an isomorphism.  
\end{proof}

The last result generalizes van Kampen's theorem. If $\{ A,B\}$ is an open covering of a locally finite $T_0$-space $X$, with $A,B$ and $A\cap B$ connected, then every digon of $X$ is contained in $A$ or $B$, so the result reduces to the classical van Kampen's theorem. Also if $K$ is a regular CW-complex covered by two connected subcomplexes $L, M$ with connected intersection, then $\x (L)$ and $\x (M)$ are open subspaces of $\x (K)$ and every digon of $\x (K)$ is in one of the subspaces. However, our result allows one to work also with  non-simplicial combinatorial decompositions $\{A,B\}$ of $\x (K)$, obtaining information on the fundamental group of $K$ from the fundamental groups of the ``discrete parts'' $A$ and $B$.

\bigskip

We finish this section with a result that characterizes posets with abelian fundamental group in terms of colorings. Given a $G$-coloring $c$ of $X$ we denote by $c^{-1}$ the $G$-coloring defined by $c^{-1}(x,y)=c(x,y)^{-1}$ for every edge $(x,y)\in \ee (X)$.

\begin{teo}
Let $X$ be a connected locally finite $T_0$-space and let $x_0\in X$. The following are equivalent:
\begin{itemize}
\item[(i)] $\pi _1 (X,x_0)$ is abelian.
\item[(ii)] For every group $G$ and every admissible connected $G$-coloring $c$ of $X$, $c^{-1}$ is an admissible and connected $G$-coloring of $X$.
\end{itemize}
\end{teo}
\begin{proof}
If $\pi_1 (X,x_0)$ is abelian and $c$ is an admissible connected $G$-coloring of $X$, then $G$ is abelian by Corollary \ref{colorepi}. Then the inverse map $G\to G$ is a homomorphism and therefore $c^{-1}$ is equivalent to $c$. In particular it is admissible and connected.
Conversely, suppose that (ii) holds. We consider two cases: when $X$ has at least one digon or when $X$ has no digon. In the first case, let $D$ be a simple digon which is the union of the chains $x=x_0\prec x_1 \prec \ldots \prec x_k=y$ and $x=x_0' \prec x_1' \prec \ldots \prec x_l'=y$. Let $g,h \in \pi _1(X,x_0)$. Let $c$ be an admissible connected $\pi _1 (X,x_0)$-coloring of $X$. Since $D$ is the diagram of a simply-connected space, by Lemma \ref{arbol} there exists a coloring $c'$ equivalent to $c$ which is trivial in $D$. We consider the coloring $c''$ obtained from $c'$ when choosing, following the notations of Definition \ref{defequi}, $g_{x_{k-1}}=g$, $g_y=hg$, all the other $g_z=1$, and $\varphi =1_{\pi_1 (X,x_0)}$. This coloring is admissible and connected and then, by hypothesis, $(c'')^{-1}$ is also admissible. The admissibility of $(c'')^{-1}$ in the digon $D$ says that $gh=hg$. Thus, $\pi _1 (X,x_0)$ is abelian.

Assume now that $X$ has no digons. In this case, by Theorem \ref{main3}, $\pi _1(X,x_0)$ is a free group. Suppose $\pi _1 (X,x_0)$ is not abelian. Then it is a free group on at least two generators. We claim that there exist two closed edge-paths (not necessarily at $x_0$ nor at the same base point) $\xi=e_0e_1\ldots e_k$, $\xi'=e_0'e_1'\ldots e_l'$ which are simple (i.e. any vertex is in at most two edges of each path) and such that $e_0$ and $e_1$ are not edges of $\xi'$, with any orientation, and $e_0'$ is not an edge of $\xi$, with any orientation. Moreover $e_0 '$ is not adjacent to $x$, the common vertex of $e_0$ and $e_1$. Since $\pi _1(X,x_0)$ is not cyclic, the underlying undirected graph of the Hasse diagram of $X$ has at least two simple cycles $\xi$, $\xi '$. If there exists a vertex $x$ of $\xi$ which is not a vertex of $\xi '$, then we take $e_0$ and $e_1$ as its adjacent edges in $\xi$ and $e_0'$ as any edge of $\xi '$ which is not in $\xi$. In the case that $\xi$ and $\xi '$ have exactly the 
same set of vertices, take any edge $e$ of $\xi '$ not in $\xi$, then in the subgraph of edges $e, e_0, e_1, \ldots , e_k$ there are three simple cycles and at least two of them have different length. In the longest, there is a vertex which is not in the other and we can reason as above.

Consider now the Dihedral group $D_3=< s,r \ | \ s^2, r^3, (rs)^2 >$. We will show that there exists an admissible and connected $D_3$-coloring $c$ of $X$ such that $c^{-1}$ is not connected.
Let $x_1$ be the vertex of $e_0$ different from $x$. If $x_1\prec x$, color $e_0$ with color $r$, if $x\prec x_1$, color it with $r^2$. Color all the remaining edges adjacent to $x$ with color $sr^2$. Color $e_0'$ with $s$, and the rest of the edges of $X$ with the trivial color $1$. This coloring $c$ is admissible since $X$ has no digons. The weight of $\xi$ is $rsr^2=sr$. Take $\gamma$ the shortest edge-path from $x_1$ to $x_1 '$, the base vertex of $\xi'$, and define the closed edge-path at $x_1$, $\widetilde{\xi '}=\gamma \xi ' \gamma ^{-1}$. Then the weight of $\widetilde{\xi '}$ is $s$ or $r sr^2 s (rsr^2)^{-1}=sr^2$, depending on if $e_0$ is in $\gamma$ or not. In any case $\{sr, s\}$ and $\{sr, sr^2\}$ are generating sets of $D_3$, which proves that $c$ is connected. On the other hand, the coloring $c^{-1}$ is not connected. It coincides with $c$ in each edge of $X$ with exception of $e_0$. Taking $g_{x}=sr^2$ and all the other $g_z$ trivial, we obtain a coloring $c'$ equivalent to $c^{-1}$ 
such that $c'(e_0)=c'(e_0')=s$ while all the other edges of $X$ are colored with $1$. Then $c'$ is not connected, and therefore neither is $c^{-1}$.
\end{proof} 

\section{Colorings and $\pi_1$  on maps}\label{maps}

In this section we characterize, in terms of colorings, the maps of posets which induce sections, epimorphisms or the trivial map between the fundamental groups.
In some cases we prove first the result for inclusions and then we achieve the general result by considering a discrete analogue of the mapping cylinder.
If $f:X\to Y$ is a map between posets, then the \textit{non-Hausdorff mapping cylinder} $\widetilde{B}(f)$ is the poset whose underlying set is the disjoint union of $X$ and $Y$ keeping the given ordering within $X$ and $Y$,
and setting $y<x$ for $x\in X$ and $y\in Y$ if $y \le f(x)$. The map $r:\widetilde{B}(f)\to Y$ which maps $x$ to $f(x)$ for every $x\in X$ and $y$ to $y$ for each $y\in Y$ is a homotopy equivalence 
(if $j:Y \to \widetilde{B}(f)$ is the canonical inclusion, the homotopy which coincides with $jr$ for $t<1$ and with $1_{\widetilde{B}(f)}$ for $t=1$ is continuous). 
This allows us to replace $f_*$ by the map $i_*$ induced by the inclusion of $X$ in the cylinder.  In \cite{Bar2, BM} we considered a slightly different version $B(f)$ of the cylinder. In fact, $\tilde B(f)=B(f^{op})^{op}$. Note that if $f:X\to Y$ is a map between locally finite $T_0$-spaces, $\widetilde{B}(f)$ is locally finite.

\begin{lema} \label{extender}
Let $A$ be a connected space corresponding to a subdiagram $D$ of the Hasse diagram of a connected locally finite $T_0$-space $X$ and let $x_0\in A$. Let $G$ be a group and $c, c'$ two admissible $G$-colorings of $A$. If $c$ extends to an admissible $G$-coloring $\widetilde{c}$ of $X$, then $c'$ extends to an admissible $G$-coloring $\widetilde{c'}$ of $X$ which is equivalent to $\widetilde{c}$.
\end{lema}
\begin{proof}
Since $c$ and $c'$ are equivalent, there exist an automorphism $\varphi : G\to G$ and a family $\{g_x\}_{x\in A}$ of elements of $G$ such that $c'(x,y)=\varphi (g_x c(x,y) g_y^{-1})$ for every $(x,y)\in \ee (A)$. Define a $G$-coloring of $X$ by $\widetilde{c'}(x,y)=\varphi (h_x \widetilde{c}(x,y) h_y^{-1})$ for every $(x,y)\in \ee (X)$, where $h_x=g_x$ if $x\in A$ and $h_x=1$ otherwise. Then $\widetilde{c'}\sim \widetilde{c}$ and it extends $c'$.
\end{proof}

Given an admissible $G$-coloring $c$ of a locally finite poset $X$ and any two elements $x,x'\in X$ such that $x\le x'$, we will denote by $c(x,x')$ the weight of any monotonic edge-path from $x$ to $x'$. If $x=x'$, then $c(x,x')=1$, the identity of $G$.

\begin{teo} \label{teoseccion}
Let $f:X\to Y$ be a continuous map between connected locally finite $T_0$-spaces and let $x_0\in X$. Then the following are equivalent

\noindent (i) The homomorphism $f_*:\pi _1 (X,x_0) \to \pi _1(Y, f(x_0))$ is a section.

\noindent (ii) For every group $G$ and every admissible connected $G$-coloring $c$ of $X$, there exist an admissible $G$-coloring $\widetilde{c}$ of $Y$ and $g_x\in G$ for each $x\in X$ such that $$\widetilde{c}(f(x),f(x'))=g_xc(x,x')g_{x'}^{-1}$$ for every edge $(x,x')\in \ee(X)$.
\end{teo}
\begin{proof}
Suppose first that the Hasse diagram of $X$ is a subdiagram of the Hasse diagram of $Y$. In this case, if the inclusion $i:X \hookrightarrow Y$ induces a section $i_* :\hh (X,x_0) \to \hh (Y,x_0)$, let $r: \hh (Y,x_0) \to \hh (X,x_0)$ be a homomorphism such that $ri_*=1_{\hh (X,x_0)}$. Let $c$ be an admissible connected $G$-coloring of $X$ and let $W:\hh (X,x_0) \to G$ be the weight map induced by $c$. Choose for each $y\in Y$ an edge path $\gamma _y$ from $x_0$ to $y$ in such a way that $\gamma _x$ is contained in $X$ for every $x\in X$. Define the $G$-coloring $\widetilde{c}$ by $\widetilde{c}(y,y')=Wr (\langle \gamma _y (y,y') \gamma _{y'} ^{-1} \rangle)$ for each $(y,y')\in \ee (Y)$. Then $\widetilde{c}$ is admissible. Moreover, if $(x,x')\in \ee (X)$, $\widetilde{c}(x,x')=Wr(\langle \gamma _x (x,x') \gamma _{x'} ^{-1} \rangle)= W(\langle \gamma _x (x,x') \gamma _{x'} ^{-1} \rangle)=w(\gamma _x)c(x,x') w(\gamma _{x'})^{-1}$. Therefore, $\widetilde{c}|_X$ is equivalent to $c$. Since $\widetilde{c}|_X$ 
extends to an admissible $G$-coloring of $Y$, by Lemma \ref{extender} so does $c$.

Suppose now $f:X\to Y$ is any continuous map between connected locally finite $T_0$-spaces such that $f_*: \pi _1 (X,x_0) \to \pi _1(Y, f(x_0))$ is a section. 
Then the inclusion $i: X \hookrightarrow \widetilde{B}(f)$ induces a section $i_*: \hh (X, x_0) \to \hh (\widetilde{B}(f), x_0)$. 
Given an admissible connected $G$-coloring $c$ of $X$, by the previous paragraph, this extends to an admissible $G$-coloring $\widetilde{c}$ of $\widetilde{B}(f)$. The restriction of this coloring to $Y$ is an admissible coloring. Let $g_x=\widetilde{c}(f(x),x)$ for every $x\in X$. The admissibility of $\widetilde{c}$ for a digon containing $f(x), f(x'), x$ and $x'$ determines the identity $$\widetilde{c}(f(x),f(x'))=g_xc(x,x')g_{x'}^{-1}.$$

Conversely, let $c$ be the standard coloring of $X$. By hypothesis there exist an admissible $\hh (X,x_0)$-coloring $\widetilde{c}$ of $Y$ and a family $\{g_x\}_{x\in X}$ satisfying the identity above.
This gives an $\hh (X,x_0)$-coloring $c'$ of $\widetilde{B}(f)$ which coincides with $c$ in $X$, with $\widetilde{c}$ in $Y$ and such that $c'(f(x),x)=g_x$ if $f(x)\prec x$. Note then that $c'(f(x),x)=g_x$ for every $x\in X$.
The coloring $c'$ is admissible since for a digon with minimum $y\in Y$, maximum $x\in X$ and containing the edges $(f(x'),x')$ and $(f(x''),x'')$ (see Figure \ref{tresromb} below) one has

$$c'(y,f(x'))c'(f(x'),x')c'(x',x)=\widetilde{c}(y,f(x'))g_{x'}c(x',x)=\widetilde{c}(y,f(x'))\widetilde{c}(f(x'), f(x))g_x=$$ $$=\widetilde{c}(y,f(x''))\widetilde{c}(f(x''), f(x))g_x=\widetilde{c}(y,f(x''))g_{x''}c(x'', x)=c'(y,f(x''))c'(f(x''), x'')c'(x'',x).$$

\begin{figure}[h] 
\begin{center}
\includegraphics[scale=0.4]{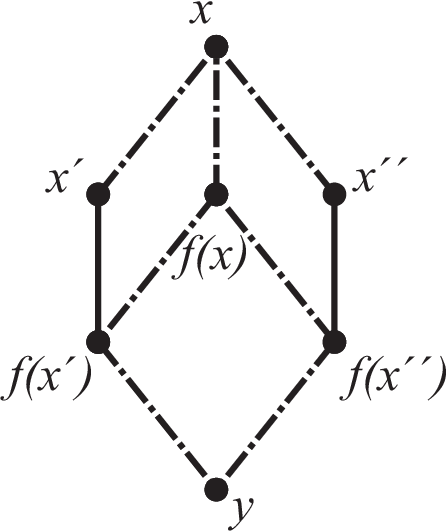}
\caption{A digon decomposed in three digons. The lines in the diagram represent monotonic paths.}\label{tresromb}
\end{center}
\end{figure}

Since the coloring $c$ extends to an admissible coloring $c'$ of $\widetilde{B}(f)$, the weight map $W_{c'}: \hh(\widetilde{B}(f),x_0) \to \hh (X,x_0)$ satisfies $W_{c'}i_*=W_c=1_{\hh (X,x_0)}$, which proves that $i_*$ is a section. Then $f_*$ is also a section.
\end{proof}

\begin{obs}
Note that if $A$ is a connected subdiagram of the Hasse diagram of a connected locally finite poset $X$, then the inclusion $A\hookrightarrow X$ induces a section in the fundamental groups if and only if for every group $G$, each admissible connected $G$-coloring of $A$ extends to an admissible $G$-coloring of $X$. This follows from Theorem \ref{teoseccion} and its proof.
\end{obs}

\begin{teo} \label{teoepi}
Let $f:X\to Y$ be a continuous map between connected locally finite $T_0$-spaces and let $x_0\in X$. Then the following are equivalent

\noindent (i) The homomorphism $f_*:\pi _1 (X,x_0) \to \pi _1(Y, f(x_0))$ is an epimorphism.

\noindent (ii) For every group $G$ and every admissible connected $G$-coloring $c$ of $Y$, the $G$-coloring of $X$ given by $$\widetilde{c}(x,x')=c(f(x),f(x'))$$ is connected.
\end{teo}
\begin{proof}
Given an admissible $G$-coloring $c$ of $Y$, the coloring of $X$ defined by the identity of (ii) is clearly admissible. Moreover, there is a commutative triangle

\begin{displaymath}
\xymatrix@C=23pt{ \hh (X,x_0) \ar@{->}^{W_{\widetilde{c}}}[dr] \ar@{->}^{f_*}[rr] & & \hh (Y, f(x_0)) \ar@{->}^{W_c}[dl] \\
								& G. \\ }
\end{displaymath}
If $f_*$ is an epimorphism and $c$ is connected, then $W_c$ is an epimorphism and therefore so is $W_{\widetilde{c}}$, which shows that $\widetilde{c}$ is connected. Conversely, if (ii) holds, then for the standard coloring $c$ of $Y$ we have that $\widetilde{c}$ is connected and then $f_*=W_{\widetilde{c}}$ is an epimorphism.
\end{proof}

\begin{obs} \label{obsepi}
By the last result, if $A$ is a connected subdiagram of the Hasse diagram of a connected locally finite poset $X$, then the inclusion $A\hookrightarrow X$ induces an epimorphism in the fundamental groups if and only if for every group $G$, each admissible connected $G$-coloring of $X$ restricts to a connected $G$-coloring of $A$.
\end{obs}

\begin{teo} \label{teotrivial}
Let $f:X\to Y$ be a continuous map between connected locally finite $T_0$-spaces and let $x_0\in X$. Then the following are equivalent

\noindent (i) The homomorphism $f_*:\pi _1 (X,x_0) \to \pi _1(Y, f(x_0))$ is the trivial map $f_*=0$.

\noindent (ii) For every group $G$ and every admissible $G$-coloring $c$ of $Y$, there exist $g_x\in G$ for each $x\in X$ and a $G$-coloring $\widetilde{c}$ of $Y$, equivalent to $c$ such that $$\widetilde{c}(f(x),f(x'))=g_xg_{x'}^{-1}$$ for every $(x,x')\in \ee (X)$.
\end{teo}
\begin{proof}
Suppose $f_*=0$ and let $c$ be an admissible $G$-coloring of $Y$. We define a $G$-coloring of $\widetilde{B}(f)$ by $c'(z,z')=c(r(z),r(z'))$ where $r:\widetilde{B}(f)\to Y$ is the retraction of the non-Hausdorff mapping cylinder onto $Y$.
Clearly, $c'$ is admissible. Since $i_*:\hh (X, x_0)\to \hh (\widetilde{B}(f),x_0)$ is trivial, by Lemma \ref{arbol} there exists a $G$-coloring $\widetilde{c}$ of $\widetilde{B}(f)$ equivalent to $c'$ which is trivial in $X$.

Since $\widetilde{c}$ is admissible, for a digon containing an edge $(x,x')\in \ee (X)$, $f(x)$ and $f(x')$, we have $\widetilde{c}(f(x),x)\widetilde{c}(x,x')=\widetilde{c}(f(x),f(x'))\widetilde{c}(f(x'),x')$. Let $g_x=\widetilde{c}(f(x),x)$.
Since $\widetilde{c}$ is trivial in $X$, we have $g_xg_{x'}^{-1}=\widetilde{c}(f(x),f(x'))$. The restriction of $\widetilde{c}$ to $Y$ is equivalent to $c'|_Y=c$ and satisfies the required identity.

Conversely, assume now that condition (ii) holds. Let $c$ be the standard coloring of $\widetilde{B}(f)$. 
Then $c|_Y$ is equivalent to some coloring $\widetilde{c}$ such that $\widetilde{c}(f(x),f(x'))=g_xg_{x'}^{-1}$ for every $(x,x')\in \ee (X)$ and some family $\{g_x\}_{x\in X}$. 
By Lemma \ref{extender}, $\widetilde{c}$ extends to an $\hh (\widetilde{B}(f), x_0)$-coloring $\widetilde{C}$ of $\widetilde{B}(f)$ equivalent to $c$. 
Define for each $x\in X$, $h_x=\widetilde{C}(f(x),x)$. Since $\widetilde{C}$ is admissible, for every $(x,x')\in \ee (X)$ we have $$h_x\widetilde{C}(x,x')=\widetilde{C}(f(x),f(x'))h_{x'}=g_xg_{x'}^{-1}h_{x'}.$$
Thus, $\widetilde{C}(x,x')=h_x^{-1}g_xg_{x'}^{-1}h_{x'}$. Choosing for every $x\in X$, $k_x=g_x^{-1}h_x$, we obtain a coloring $c'$ of $\widetilde{B}(f)$, equivalent to $\widetilde{C}$, 
and such that $c'|_X$ is trivial. Therefore we obtain a coloring of $\widetilde{B}(f)$ equivalent to the standard coloring $c$ which is trivial in $X$. 
Then $1_{\hh (\widetilde{B}(f),x_0)}=W_{c}=\varphi W_{c'}$ for some $\varphi \in \textrm{Aut}(\hh (\widetilde{B}(f),x_0))$. Hence, if $\xi$ is a closed edge-path in $X$, 
the class $i_*(\langle \xi \rangle)$ of $\xi$ in $\hh (\widetilde{B}(f),x_0)$ is $\langle \xi \rangle=W_c(\langle \xi \rangle)= \varphi W_{c'} (\langle \xi \rangle)=1$. This says that $i_*=0$ and then $f_*=0$. 
\end{proof}

\begin{obs}
When $A$ is a connected subdiagram of a connected locally finite poset $X$, the inclusion $i:A\hookrightarrow X$ induces the trivial homomorphism between the fundamental groups if and only if for every group $G$, 
each admissible $G$-coloring of $X$ is equivalent to a coloring which is trivial in $A$. This follows directly from Lemma \ref{arbol} and the last theorem.
\end{obs}

Recall that a digon is called simple if it consists of two monotonic edge-paths from a point $x$ to a point $y$ which have no common vertex other than $x$ and $y$. 

\begin{prop}
Let $Y$ be a connected locally finite $T_0$-space, $y_0\in Y$ and let $(a,b)\in \ee (Y)$. Let $X$ be the space corresponding to the subdiagram of $Y$ obtained when removing the edge $(a,b)$.
If $(a,b)$ is contained in a simple digon, the map $i_*: \pi _1(X,y_0) \to \pi _1 (Y,y_0) $ induced by the inclusion is an epimorphism. 
\end{prop}
\begin{proof}
By Remark \ref{obsepi} it suffices to check that any admissible connected $G$-coloring $c$ of $Y$ restricts to a connected coloring of $X$.
Note that $X$ is connected. If $\xi$ is a closed edge-path at $y_0$ in $Y$, then there is another closed edge-path with the same weight as $\xi$ and which does not contain the edge $(a,b)$.
We can just avoid edge $(a,b)$ using the remaining edges of the simple digon. Thus $c|_X$ is also connected. 
\end{proof}

If $X$ is simply-connected and we add an edge which appears as part of a simple digon, the new space is also simply-connected. However, if the edge is not contained in a simple digon, the fundamental group of the new space is $\mathbb{Z}$.

\section{Boards on surfaces}

In the previous sections we used colorings to study problems of topological nature. In this section we exhibit an application in a different direction. Consider the following elementary combinatorial problem. Let $n$ and $m$ be positive integers and suppose we have an $n\times m$ rectangular board. The edges of the squares in the board are colored either with blue or with red, and one such coloring is called \textit{valid} if for each square of the board exactly $0$, $2$ or $4$ of its edges are colored with blue. A possible \textit{move} is to pick a vertex of the board and change the colors of all the (two, three or four) edges incident to that vertex, blue by red and red by blue. Prove that if $c$ and $c'$ are two valid colorings of the board, then it is possible to obtain $c'$ from $c$ by performing a finite sequence of moves.

\begin{figure}[h]
\begin{minipage}{6cm}
\vspace{.3cm}
\center{\includegraphics[scale=0.5]{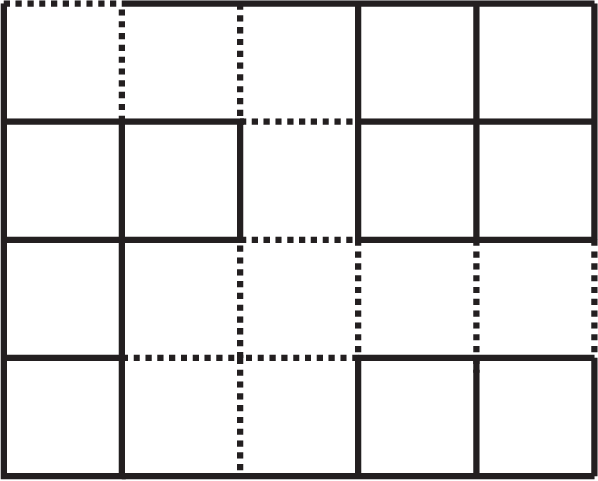}}
\end{minipage}
\begin{minipage}{6cm}
\center{\includegraphics[scale=0.5]{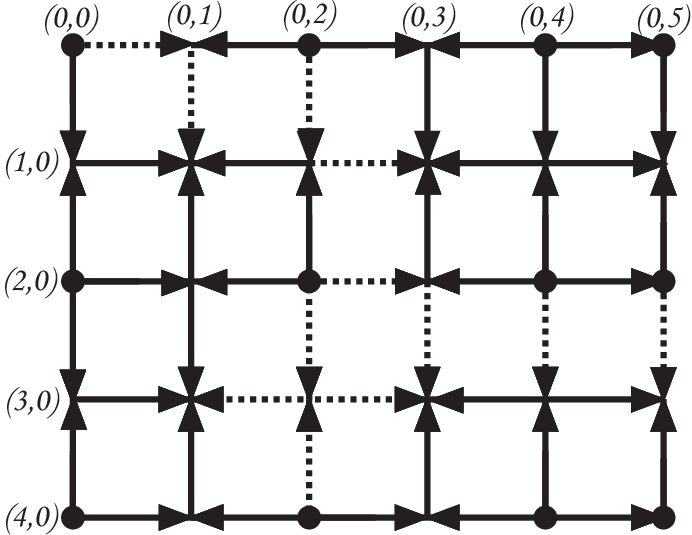}}
\end{minipage}
\caption{On the left, a $4\times 5$ board and a valid coloring. The solid edges represent color blue and the dotted edges red. On the right, the Hasse diagram of the poset $I_4\times I_5$ in which the orientation of the edges is indicated with an arrow. The corresponding admissible coloring is represented using solid edges for the identity of $\mathbb{Z}_2$ and dotted edges for the non-trivial element. \protect\label{figtablero}}
\end{figure}

We give a solution using the methods described in this paper. Let $I_n$ be the poset $0<1>2<3>\ldots n$. Then $I_n$ and $I_m$ are contractible and therefore, so is the product $I_n \times I_m$. In particular any two admissible $\mathbb{Z}_2$-colorings of $I_n\times I_m$ are equivalent. The Hasse diagram of $I_n\times I_m$ is an $n\times m$ board where the edges of the diagram coincide with edges of the squares (see Figure \ref{figtablero}). A $\mathbb{Z}_2$-coloring is a coloring of the edges with colors $0=$blue and $1=$red. The admissibility of the coloring is equivalent to the validity. Finally, the equivalence of $\mathbb{Z}_2$-colorings is the same as the existence of moves taking one coloring to the other.

Suppose now that we have a cylindrical board, obtained from the $n\times m$ board by identifying the top edge of each square in the first row with the bottom edge of the square in the last row and the same column. Note that the notions of valid colorings and moves still make sense. In this case there exist two valid colorings such that none of them can be obtained from the other by performing allowed moves. However, given any three valid colorings, there are two of them which are related by a sequence of moves.

To see this consider, when $n\ge 4$ is even, the poset $C_n$ which is obtained from $I_{n}$ by identifying $0$ and $n$. It is the poset $0<1>2<\ldots <n-1>0$ (see Figure \ref{figc89}).

\begin{figure}[h]
\begin{minipage}{8cm}
\center{\includegraphics[scale=0.6]{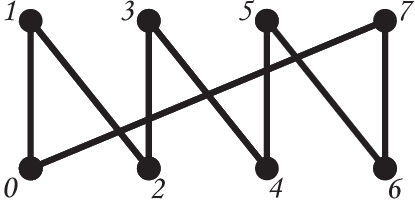}}
\end{minipage}
\begin{minipage}{5cm}
\center{\includegraphics[scale=0.6]{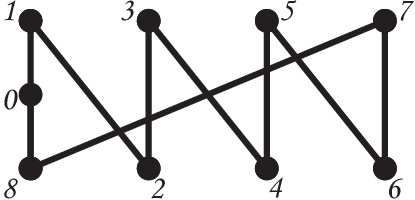}}
\end{minipage}
\caption{$C_8$ and $C_9$.\protect\label{figc89}}
\end{figure}

The fundamental group of $C_n$, and also of $C_n\times I_m$ is infinite cyclic. As in the first case, the edges of the Hasse diagram of $C_n\times I_m$ are in correspondence with edges of squares in the board and admissibility equals validity of the coloring. Since $\mathbb{Z}_2$ is a quotient of $\mathbb{Z}$, by Theorem \ref{main1}, there exists an admissible and connected $\mathbb{Z}_2$-coloring $c$ of $C_n\times I_m$. The coloring $c$ is connected and therefore it cannot be equivalent to the trivial coloring. In this way we obtain non-equivalent colorings of the board. In the case that $n\ge 5$ is odd, we define $C_n$ again by identifying $0$ and $n$ in $I_n$. It is the poset $0<1>2<\ldots >n-1<0$. Now the height of $C_n$ is two but it still has fundamental group isomorphic to $\mathbb{Z}$. The vertices and edges in the Hasse diagram of $C_n\times I_m$ are still in correspondence with vertices and edges in the cylindrical board. It is still true that validity of a coloring is equivalent to admissibility 
although this is a little harder to see. Therefore, also when $n$ is odd, there are two colorings of the board where one cannot be obtained from the other and they correspond to a connected and a non-connected $\mathbb{Z}_2$-coloring of $C_n\times I_m$. Now, if $c$ is a non-connected admissible $\mathbb{Z}_2$-coloring of a poset $X$, it induces the trivial weight $W=W_c:\hh (X,x_0)\to \mathbb{Z}_2$ and, by the proof of Lemma \ref{arbol}, it is equivalent to the trivial coloring. Hence, two non-connected $\mathbb{Z}_2$-colorings of $C_n\times I_n$ are equivalent. On the other hand, there exists a unique normal subgroup $N\triangleleft \mathbb{Z}$ such that $\mathbb{Z}/ N$ is isomorphic to $\mathbb{Z}_2$. It follows from Theorem \ref{main1} that any two connected admissible $\mathbb{Z}_2$-colorings of $C_n\times I_m$ are equivalent. Finally we deduce that in any three valid colorings of the cylindrical board, there are two such that one can be obtained from the other by a sequence of allowed moves.

Of course these results can be applied in other examples. The analysis of the toric board, obtained by identifying left and right edges of the rectangular board as well as the top and the bottom, is similar to the cylindrical one but considering the poset $C_n\times C_m$. In this case it is not longer true that in any three valid colorings there are two equivalent since there are two different subgroups of $\mathbb{Z}\times \mathbb{Z}$ of index $2$.

\end{document}